\documentclass[12pt,reqno]{amsart}
  \usepackage{latexsym} 
  \usepackage[all]{xy}
  \usepackage{amsfonts} 
  \usepackage{amsthm} 
  \usepackage{amsmath} 
  \usepackage{amssymb}
  \usepackage{pifont}  
  \usepackage{enumerate}
  \usepackage{dcolumn}
  \usepackage{comment}
  \usepackage{hyperref}  
\newcolumntype{2}{D{.}{}{2.0}}
  \xyoption{2cell}

 \usepackage{tikz}
\usetikzlibrary{arrows}

  \def\<{{\langle}} 
  \def\>{{\rangle}}

  \def\note#1{{}}

  \def\note#1{} 

  \def\beq{\begin{equation}} 
  \def\eeq{\end{equation}}

  \def\im{{\rm Im}}

  \newcounter{zlist} 
  \newenvironment{zlist}{\begin{list}{(\arabic{zlist})}{ 
  \usecounter{zlist}\leftmargin2.5em\labelwidth2em\labelsep0.5em 
  \topsep0.6ex
  \parsep0.3ex plus0.2ex minus0.1ex}}{\end{list}}

  \newcounter{blist} 
  \newenvironment{blist}{\begin{list}{(\alph{blist})}{ 
  \usecounter{blist}\leftmargin2.5em\labelwidth2em\labelsep0.5em 
  \topsep0.6ex 
  \parsep0.3ex plus0.2ex minus0.1ex}}{\end{list}} 

  \newcounter{rlist} 
  \newenvironment{rlist}{\begin{list}{(\roman{rlist})}{ 
  \usecounter{rlist}\leftmargin2.5em\labelwidth2em\labelsep0.5em 
  \topsep0.6ex 
  \parsep0.3ex plus0.2ex minus0.1ex}}{\end{list}}

\def\stac#1{\raise-.2cm\hbox{$\stackrel{\displaystyle\otimes}{\scriptscriptstyle{#1}}$}}

\def\cten#1{\raise-.2cm\hbox{$\stackrel{\displaystyle\widehat{\otimes}}
{\scriptscriptstyle{#1}}$}}

  \def\Label#1{\label{#1}\ifmmode\llap{[#1] }\else 
  \marginpar{\smash{\hbox{\tiny [#1]}}}\fi} 
  \def\Label{\label}

  \newtheorem{proposition}{Proposition}[section]
  \newtheorem{lemma}[proposition]{Lemma} 
  \newtheorem{corollary}[proposition]{Corollary} 
  \newtheorem{theorem}[proposition]{Theorem} 
  
\theoremstyle{definition} 
  \newtheorem{definition}[proposition]{Definition}
    
  \newtheorem{example}[proposition]{Example}

  \theoremstyle{remark} 
  \newtheorem{remark}[proposition]{Remark}

  \newcounter{c} 
  \renewcommand{\[}{\setcounter{c}{1}$$} 
  \newcommand{\etyk}[1]{\vspace{-7.4mm}$$\begin{equation}\Label{#1} 
  \addtocounter{c}{1}} 
  \renewcommand{\]}{\ifnum \value{c}=1 $$\else \end{equation}\fi} 
   \numberwithin{equation}{section}

\def\QQ{{\mathbb Q}}

\def\ZZ{{\mathbb Z}}

\newcommand{\gG}{\mathrm{G}}
\newcommand{\hH}{\mathrm{H}}

\newcommand{\qQ}{\mathrm{Q}}

\newcommand{\tT}{\mathrm{T}}

\newcommand{\Cc}{\mathcal{C}}

\def\*C{{}^*\hspace*{-1pt}{\Cc}}

\def\text#1{{\rm {\rm #1}}}

 \def\1{\mathbf{1}}

\def\lto{\longmapsto}
\def\lra{\longrightarrow}

\def\1\mathbf{1}
\def\bfa{\mathbf{a}}
\def\bfb{\mathbf{b}}
\def\bfc{\mathbf{c}}
\def\bfq{\mathbf{q}}

\def\|#1{\overline{#1}}
\def\Abs{\mathrm{Abs}}

\begin{document}

\title{From pre-trusses to skew braces}

\author{Tomasz Brzezi\'nski}

\address{
Department of Mathematics, Swansea University, 
Swansea University Bay Campus,
Fabian Way,
Swansea,
  Swansea SA1 8EN, U.K.\ \newline \indent
Chair of Algebra and Geometry, Faculty of Mathematics, University of Bia{\l}ystok, K.\ Cio{\l}kowskiego  1M,
15-245 Bia\-{\l}ys\-tok, Poland}

\email{T.Brzezinski@swansea.ac.uk}

\author{Stefano Mereta}

\address{
  Department of Mathematics, Swansea University, 
  Swansea University Bay Campus,
  Fabian Way,
  Swansea SA1 8EN, U.K. \newline \indent
  Institut Fourier,
  Universit\'e Grenoble Alpes,
  100 rue des maths,
  38610 Gi\'eres, France}

\urladdr{https://sites.google.com/view/stefanomereta/home}

\email{988913@Swansea.ac.uk, stefano.mereta@univ-grenoble-alpes.fr}

\author{Bernard  Rybo{\l}owicz}

\address{
Department of Mathematics, Swansea University, 
Swansea University Bay Campus,
Fabian Way,
Swansea,
  Swansea SA1 8EN, U.K.}

\urladdr{https://sites.google.com/view/bernardrybolowicz/}
\email{Bernard.Rybolowicz@swansea.ac.uk}

\subjclass[2010]{16Y99; 08A99}

\keywords{Pre-truss; near-truss; heap; skew brace; near-ring}

\begin{abstract}
The notion of a pre-truss, that is, a set that is both a heap and a semigroup is introduced. Pre-trusses themselves as well as  pre-trusses in which one-sided or two-sided distributive laws hold are studied. These are termed near-trusses and skew trusses respectively. Congruences in pre-trusses are shown to correspond to paragons defined here as sub-heaps satisfying particular closure property. Near-trusses corresponding to skew braces and near-rings are identified through their paragon and ideal structures. Regular elements in a pre-truss are defined leading to the notion of a (pre-truss) domain. The latter are described as quotients by completely prime paragons, also defined hereby. Regular pre-trusses as domains that satisfy the Ore condition are introduced and the pre-trusses of fractions are defined. In particular, it is shown that near-trusses of fractions without an absorber correspond to skew braces.
\end{abstract}    
\date\today
\maketitle

\section{Introduction}
In the 1920s H.\ Pr\"ufer and R.\ Baer defined a heap as an algebraic system consisting of a set with a ternary operation which fulfils conditions that allow one to associate an isomorphic group to every element; conversely, every group gives rise to a heap by taking  the operation $(a,b,c)\mapsto ab^{-1}c$  (see \cite{Bae:ein} and \cite{Pru:the}).
In 2007 W.\ Rump introduced braces as algebraic systems corresponding to solutions of set-theoretic  Yang-Baxter equations \cite{Rum:bra}. A brace is a triple $(G,+,\cdot)$ where $(G,+)$ is an Abelian group, $(G,\cdot)$ is a group and the following distributive law holds, for all $a,b,c\in G$,
$$
a\cdot(b+c)=a\cdot b-a+a\cdot c;
$$
see \cite{CedJes:bra}. Through their connection with set-theoretic Yang-Baxter equations, braces have become an intensive field of studies.  In particular it has been shown that a brace allows one to construct a non-degenerate involutive set-theoretic solution of Yang-Baxter equation (see for example, \cite{CedJes:bra}, \cite{Rum:bra}, \cite{CIS:LB} and \cite{A:YBB}). In 2017 L.\ Guarnieri and L.\ Vendramin introduced the notion of a skew brace. This is a generalisation of a brace in which $(G,+)$ is not required to be Abelian \cite{GuaVen:ske}, it has been shown to  correspond  bijectively to non-degenerate solutions of the set-theoretical Yang-Baxter equation (see e.g.\ \cite{Bachiller}). 
In recent years there has been a vast progress in the research on solutions of set-theoretical Yang-Baxter equations,  
but, even though we know that every skew brace provides us with such a solution, it is not an easy task to construct skew braces (for a list of problems on skew braces and a literature review see \cite{Ven:pro}). In 2018 the first author observed that it is possible to unify the distributive laws of rings and braces in a single more general algebraic structure which he called truss \cite{Brz:tru}. A left skew truss $T$ is a heap $(T,[-,-,-])$  with an additional  binary operation $\cdot:T\times T\to T$ which is associative and which distributes over the ternary operation from the left, i.e., for all $a,b,c,d\in T$, 
$$
a\cdot [b,c,d]=[a\cdot b,a\cdot c,a\cdot d].
$$
Every skew brace can be associated with an appropriate skew left truss: in this text we will call such trusses brace-type trusses.  This leads to the main questions that motivated the present article. What are exactly brace-type trusses? How to construct them starting from a not necessarily brace-type truss? When is such a construction possible? In the paper we present two approaches to answer questions of this kind. The first approach is to take quotients of trusses by some special congruence and the second one relies on a localisation procedure. The paper is organised as follows. 

Section~\ref{sec:2} contains definitions and facts about near-rings, skew braces and heaps which we believe to be necessary to make the paper self-contained. The section concludes with Lemma~\ref{lem:map} which describes fully all equivalence classes for a sub-heap relation $\sim_S$ as mutually isomorphic heaps with an explicitly given isomorphism in each case.

Section \ref{sec:3} starts with the introduction of pre-trusses, near-trusses and skew trusses. A pre-truss is a heap with an additional semigroup operation. A near-truss is a pre-truss in which the semigroup operation  distributes over the ternary operation from the left. The best known examples of these objects are near-trusses with left absorbers associated with near-rings (see Example~\ref{ex:near-ring}) or unital near-trusses which can be associated to recently introduced skew-rings \cite{Rum:set} (see Example~\ref{ex:skew-ring}). The notion of near-truss was introduced in \cite[Definition 2.1]{Brz:tru} under the name of a skew left truss; the present terminology is intended to be coherent with that of the near-ring theory. Another example of a near-truss which is of particular interest is a near-truss associated with a skew brace (see Example~\ref{ex:skew-brace}); these near-trusses are said to be brace-type. Finally, a skew truss is a near-truss for which the right distributive law holds.
The first part of Section~\ref{sec:3} is focused on the characterisation of algebraic structures that correspond to congruences in a pre-truss. For that, we give the definition of a paragon as a normal sub-heap, the equivalence classes of which have a particular closure property, and in Theorem~\ref{thm.par} we show that paragons fully describe all the congruences in a pre-truss. We conclude this theorem with Corollary~\ref{cor:nrec} and Corollary~\ref{cor.brace.par} which tell us that congruence equivalence classes in near-rings and skew braces are in fact paragons in the associated near-trusses. After that, we introduce the definition of an ideal to determine whether a unital near-truss is associated with a skew brace or a near-ring, Proposition~\ref{lem:sim-rb}. 

Combining the most natural concept of a maximal paragon with the analysis of the ideal structure of a truss we give a full description of those paragons whose quotient is a brace-type near-truss. More precisely in Theorem~\ref{thm:main3} we show that a quotient near-truss is brace-type if and only if  all equivalence classes are not subsets of any ideal in a near-truss. We conclude this section with two examples  of paragons that fulfil the hypothesis of the theorem.

	Section~\ref{sec:4} focuses on the study of domains. We start with the definition of a regular element in a pre-truss, then we define a domain as a pre-truss for which all elements except the absorbers are regular. In Lemma \ref{lem:cancel} we justify the definition by showing that domains are exactly pre-trusses for which the cancellation property holds. After that, we introduce the notion of completely prime paragon. 
	Since the definition of a completely prime ideal in a ring depends on an absorber and near-trusses do not necessarily have an absorber, one should expect that the definition of a completely prime paragon does not depend on it. Therefore, we fix the absence of an absorber by using ideals in the quotient, see Definition~\ref{def:prime}.  The most important result of this section is Theorem~\ref{thm:dom-prime} stating that the quotient of a pre-truss by a paragon is a domain if and only if the paragon is completely prime. We conclude this section with an example of a completely prime paragon in the truss of polynomials with integer coefficients summing up to an odd number.
	
	The aim of Section~\ref{sec:5} is to devise a method of constructing brace-type near-trusses by localisation. We start the section with  Definition~\ref{def:reg} of a left regular pre-truss. Then we describe localisation of pre-trusses. Perhaps, the most important result of this section is Corollary \ref{cor:qb} which states that if we localise a regular near-truss with no absorbers  we will obtain a skew brace. This is supplemented with an example: the localisation of a non-commutative truss of square integer matrices with odd diagonal and even off-diagonal entries.

\section{Near-rings, skew-rings, skew braces and heaps}\label{sec:2}

In this section we gather preliminary information and fix the notation and conventions on near-rings, skew braces and heaps.

\subsection{Near-rings, skew-rings and skew braces}\label{sec.near-ring} A \textbf{
 near-ring} (see \cite{Pil:nea}) is a set $N$ with two associative binary operations $+,\cdot$, such that $(N,+)$ is a group and, for all $n,m,m'\in N$, 
$$
n(m+m')=nm+nm'.
$$ 
Analogously to the case of rings a \textbf{near-field} is a near-ring such that $(N\setminus\{0\},\cdot)$ is a group, where $0$ is the neutral element for $+$. 

A \textbf{homomorphism of near-rings} is a function $f:N\to N'$  that  commutes with both near-ring operations, that is, for all $a,b\in N$,
$$
f(ab)=f(a)f(b)\quad \&\quad f(a+b)=f(a)+f(b).
$$

A \textbf{skew-ring} \cite[Definition~3 \& Corollary]{Rum:set} is a triple $(B,+,\cdot)$, where $(B,+)$ is a group  $(B,\cdot)$ is a monoid and the following distributive law holds 
$$
a(b+c)=ab-a+ac, 
$$
for all $a,b,c\in B$.
A skew-ring $(B,+,\cdot)$ in which $(B,\cdot)$  is a group is called a \textbf{skew left 
brace} \cite{GuaVen:ske}. 
A \textbf{homomorphism of skew braces} is a function that commutes with both group operations. A close connection between skew left braces and near-rings is revealed in \cite[Proposition~2.20]{SmoVen:skew}, which states that any construction subgroup of a near-ring is a skew left brace. In what follows, we drop the adjective `left', and hence skew brace means skew left brace\footnote{Obviously, `right' versions of all the notions discussed in this text can be defined and developed symmetrically, and in fact in \cite{Rum:set} Rump gives the definition of a skew-ring  in the right-sided convention.}. 
An \textbf{ideal} in a skew brace $B$ is a subset $B'\subset B$ such that $(B',+)$ is a normal subgroup, $a B'=B'a$ and $ab-a\in B'$, for all $a\in B$ and $b\in B'$.

\subsection{Heaps} A \textbf{heap} is a set $H$ together with a ternary operation,
 $$
[-,-,-]:H\times H\times H\lra H,
$$ 
that is associative and satisfies Mal'cev identities. Explicitly this means that, for all $a_{1},a_{2},a_{3},a_{4},a_{5}\in H$,
$$
[a_{1},a_{2},[a_{3},a_{4},a_{5}]]=[[a_{1},a_{2},a_{3}],a_{4},a_{5}] \quad \& \quad 
[a_{1},a_{1},a_{2}]=a_{2}=[a_{2},a_{1},a_{1}].
$$ 
These conditions imply that, for all $a_{1},a_{2},a_{3},a_{4},a_{5}\in H$,
\begin{equation}\label{assoc+}
[[a_{1},a_{2},a_{3}],a_{4},a_{5}] =[a_{1},[a_{4},a_{3},a_{2}],a_{5}]. 
\end{equation}
We say that $H$ is an \textbf{Abelian heap} if, for all $a,b,c\in H$, $[a,b,c]=[c,b,a]$. 

A \textbf{homomorphism of heaps} is a map that commutes with the ternary operations, that is, $f:H\lra H'$ is a heap morphism if, for all $a,b,c\in H$, 
$$f([a,b,c])=[f(a),f(b),f(c)].
$$ 

Every non-empty heap can be associated with a group by fixing the middle entry of the ternary operation, that is, for all $a\in H$, $+_{a}:=[-,a,-]$ is a group operation on $H$. This group 
is called a \textbf{retract} of $H$ at $a$ and is denoted by $\gG(H;{a})$. Retracts at two different elements are isomorphic. Starting with a group $G$ one can assign a heap to it  by setting $[a,b,c]:=ab^{-1}c$, for all $a,b,c\in G$. This \textbf{heap associated to a group} $G$ will be denoted by $\hH(G)$. 

A subset  $S$ of a heap $H$ is a \textbf{sub-heap} if it is closed under the heap operation of $H$. 
A non-empty sub-heap $S$ of a heap $H$ is said to be \textbf{normal} if there exists $e\in S$  such that, for all $a\in H$ and $s\in S$, there exists $t\in S$ such that 
$[a,e,s]=[t,e,a].$ This is equivalent to say that for all $a\in H$ and $e,s\in S$, $[[a,e,s],a,e]\in S$.
Every sub-heap of an Abelian heap is normal. The retract of a normal sub-heap at an element $e$ is a normal subgroup of the retract of the heap at the same element $e$. Furthermore, for any heap homomorphism $f:H\lra H'$ and any $b\in \im f$, $f^{-1}(b)$ is a normal sub-heap of $H$; see e.g.\  \cite[Lemma~2.12]{Brz:par}.

If $S$ is a sub-heap of $H$, then the relation $\sim_{S}$ on $H$ given by 
$$
a\sim_{S}b \iff \exists{s\in S}\; [a,b,s]\in S\iff \forall{s\in S}\; [a,b,s]\in S
$$
is an equivalence relation. 
The set of equivalence classes is denoted by $H/S$. The equivalence class of any $s\in S$ is equal to $S$. If $S$ is a normal sub-heap, then $\sim_{S}$ is a congruence and thus the canonical map $\pi:H\lra H/S$ is a heap epimorphism; see \cite[Proposition~2.10]{Brz:par}.

The following lemma summarises properties of the sub-heap equivalence relation and gives an explicit description of all equivalence classes and relations between them.

\begin{lemma}\label{lem:map}
Let $S$ be a non-empty sub-heap of $(H,[-,-,-])$, and consider the sub-heap relation $\sim_S$. 
 \begin{zlist}
  \item For all $a,b\in H$, define the {\rm\bf translation map}:
 \begin{equation}\label{trans}
 \tau_a^b: H\lra H, \qquad z\lto [z,a,b].
 \end{equation}
 \begin{rlist}
 \item The map $\tau_a^b$ is an isomorphism of heaps.
 \item The equivalence classes of $\sim_S$ are related by the formula:
 $$
 \bar{b} = \tau_a^b(\bar{a}) = \{[z,a,b]\; |\; z\sim_S a\}.
 $$
\item For all $e\in S$ and $a\in H$, set
$
S_e^a := \tau_e^a(S).
$
Then $\bar{a} = S_e^a$.
 \end{rlist}
 \item For all $a\in H$, the equivalence class $\bar{a}$ is a sub-heap of $H$. Furthermore, if $S$ is a normal sub-heap of $H$, then so are the $\bar{a}$.
 \item Equivalence classes of $\sim_S$ are mutually isomorphic as heaps.
 \item For all $a\in H$, the sub-heap equivalence relation $\sim_S$ coincides with the sub-heap equivalence relation $\sim_{\bar{a}}$. Consequently $H/S = H/\bar{a}$.
 \end{zlist}
 \end{lemma}
 
 \begin{proof}
 (1) (i) First we need to check that $\tau_a^b$ preserves the ternary operation. Using the associativity and Mal'cev identities, we can compute, for all $z,z',z'' \in H$,
 $$
 \begin{aligned}
\left[\tau_a^b(z),\tau_a^b(z'),\tau_a^b(z'') \right] &=\left[ [z,a,b],[z',a,b],[z'',a,b] \right]
  = \left[z,a, \left[b,[z',a,b],[z'',a,b]\right]\right]  \\
 & = \left[z,a, \left[[b,b,a],z',[z'',a,b]\right]\right] \qquad \mbox{(by equation \eqref{assoc+})}\\
    & = \left[[z,z',z''],a,b]\right]    = \tau_a^b([z,z',z'']).
\end{aligned}
 $$
 Therefore, the $\tau_a^b$  preserve ternary operations and thus each one of them is a homomorphism of heaps. The inverse of $\tau_a^b$ is $\tau_b^a$.
 
 (1)(ii) Assume that $z\sim_S a$, that is, that $[z,a,s]\in S$, for all $s\in S$. If $z'=\tau_a^b(z) = [z,a,b]$, then
 $
 [z',b,s] = [z,a,s],
 $
 by the associativity and the Mal'cev property. Hence $z'\sim_S b$, that is, $\tau_a^b(\bar{a})\subseteq \bar{b}$. On the other hand, if $z'\in \bar{b}$, then set $z=\tau_b^a(z') = [z',b,a]$. Since $\tau_b^a$ is the inverse of $\tau_a^b$, $z' = \tau_a^b(z)$. Furthermore, for all $s\in S$,
 $[z,a,s] = [z',b,s]$,
 and so $[z,a,s]\in S$, since $z'\sim_S b$. This proves the second inclusion $\bar{b} \subseteq \tau_a^b(\bar{a})$, and hence the required equality.
 
  Assertion (1)(iii) follows by 1(ii) and the fact that $\bar{e}=S$.
 
 Statement (2) follows by (1) and the observation that heap isomorphisms preserve the normality. Statement (3) is a straightforward consequence of (1) and (2).
 
 (4)  Using (1)(iii)  we can argue as follows: $b\sim_Sc$ if, and only if, there exist $s,s'\in S$ such that $[b,c,s]=s'$. This is equivalent to the equality $[[b,c,s],e,a] = [s',e,a]$, for any $a\in H$ and $e\in S$, which, by associativity, is equivalent to $[b,c,[s,e,a]] = [s',e,a]$. The fact that $\bar{a} = S_e^a$  implies that $b\sim_{\bar{a}} c$. 
 \end{proof}

 \section{Quotient pre-trusses, near-trusses and skew braces}\label{sec:3}
 The aim of this section is to characterise heaps with an additional monoid operation that yield skew braces. Let us first introduce the appropriate terminology.
 
 \begin{definition}\label{def.truss}~
 
 \begin{zlist}
 \item A \textbf{pre-truss} is a heap $(T, [-,-,-])$ together with an associative binary operation (denoted by juxtaposition of elements or by $\cdot$). 
 \item A pre-truss $T$ satisfying the left distributive law:
 $$ 
a[b,c,d]=[ab,ac,ad], \qquad \mbox{for all $a,b,c,d\in T$},
$$
is called a \textbf{near-truss}.
\item A near-truss $T$ satisfying the right distributive law
$$ 
 [b,c,d]a=[ba,ca,da], \qquad \mbox{for all $a,b,c,d\in T$},
$$ 
is called a \textbf{skew truss}.
\item A skew truss such that the underlying heap is Abelian  is called a \textbf{truss}.
\end{zlist}
Every one of the above notions is said to be \textbf{unital} provided the binary operation has an identity (denoted by 1). 

A homomorphism of (pre-, near-, skew) trusses is a homomorphism of heaps that is also a homomorphism of semigroups (or monoids in the unital case).
\end{definition}

It is clear from this definition that the image of a homomorphism of (pre-, near-, skew) trusses  is itself a  (pre-, near-, skew) truss. 

\begin{remark}
Except for a pre-truss all the notions listed in Definition~\ref{def.truss} have been introduced in \cite{Brz:tru} and \cite{Brz:par}. Note, however, that the terminology introduced there  was motivated by braces, and thus what we call a near-truss here was named  a left skew truss there. In this paper we are adopting a terminology more aligned with the ring (or near-ring) theory one. Of course, a right distributive  version of a near-truss can be considered, but in line with the convention of Section~\ref{sec.near-ring} we only consider the left distributive version (with no qualifier).
\end{remark}

A \textbf{left (resp.\ right) absorber} is an element $a$ of a pre-truss $T$ such that, for all $t\in T$, $ta=a$ (resp.\ $at=a$). We say that $a$ is \textbf{an absorber} if it is a left and right absorber.  It is worth noting that if a pre-truss $T$ has both a left and a right absorber, then they necessarily coincide, in particular an absorber is unique.  We denote by $T^{\Abs}:= T\setminus\{a\}$, if $a$ is the unique absorber with tacit understanding that  $T^{\Abs}= T$ when $T$ has no absorbers. Furthermore, since homomorphisms of pre-trusses preserve multiplication, if $f:T\lra T'$ is a morphism and $e$ is a left (resp.\ right) absorber in $T$, then $f(e)$ is a left (resp.\ right) absorber in the pre-truss $f(T)$.

\begin{example}
	If $T$ is a truss  that has an absorber, then $T$ is a \textbf{ring-type truss}. This means that by taking the retract of $T$ at the absorber, say $0$,  we obtain  a ring $(T, +_{0},\cdot)$. 
	
	Conversely, if $R$ is a ring then one can associate to it the truss $(\hH(R),\cdot)$ with absorber $0$. This truss is denoted by $\tT(R).$ If $R$ is unital, then $\tT(R)$ is unital. Observe that if we start with a ring $R$, we assign to it the truss $\tT(R)$ and then take the retract we necessarily  obtain $R$ again, since  the absorber is unique.
\end{example}

\begin{example}\label{ex:near-ring}
	Let $T$ be a near-truss such that there exists a left absorber $e$. Then  a near-ring can be associated to $T$ by taking the retract of the heap $T$ at $e$ to obtain $(T,+_{e},\cdot)$. We call such $T$ a \textbf{ring-type near-truss}
	
		Conversely, if $N$ is a near-ring then one can associate to it the near-truss $(\hH(N),\cdot)$ which we will denote by $\tT(N)$. In contrast to rings, since left absorbers are not unique, if one associates a near-truss $\tT(N)$ to $N$ and then take the retract at a left absorber, then not necessarily one obtains $N$. 
\end{example}

\begin{example}\label{ex:skew-ring}
Let $T$ be a unital near-truss. Then  a skew-ring can be associated to it by taking the retract of the heap $T$ at the identity $1$ of the multiplication, that is $(T,+_1,\cdot)$ is a skew-ring.

Conversely, if $B$ is a skew-ring, then one can assign to it the unital near-truss $(\hH(B),\cdot)$ which we will denote by $\tT(B)$. Observe that if we start with a skew-ring $B$, we assign to it the truss $\tT(B)$ and then take the retract at the identity, we obtain the same skew-ring as identity is unique.
\end{example}

Recall from \cite{Rum:set} that an element $u$ in a skew-ring $B$ is called a \textbf{unit} if, 
for all $a\in B$, 
$$
a\cdot u = a+u+a.
$$

\begin{lemma}\label{lem:skew-ring}
Units in a skew-ring $B$ are in one-to-one correspondence with left absorbers in the  associated unital near-truss $\tT(B)$. 
\end{lemma}
\begin{proof}
The correspondence is given by 
$
u = [1,e,1].
$
That is, $u$ is a unit in $B$ (resp.\ absorber in $\tT(B)$)  provided $e$ is an absorber in $\tT(B)$ (resp.\ unit in $B$).
\end{proof}

\begin{remark}\label{rem:skew-ring}
Combining Example~\ref{ex:near-ring} with Example~\ref{ex:skew-ring} and Lemma~\ref{lem:skew-ring} we are led to the correspondence between skew-rings with units and unital near-rings. If $u$ is a unit in a skew-ring $B$, then $(T(B), +_{[1,u,1]}, \cdot , 1)$ is a unital near-ring, and vice versa, if  $(N,+,\cdot, 1)$ is a unital near-ring with zero $e$, then $(T(N), +_1, \cdot)$ is a skew-ring with unit $[1,e,1] = -e$ (cf.\ \cite[Example~3]{Rum:set}). This correspondence is seemingly different from the one described in 
\cite[Proposition~2]{Rum:set} as it changes the additive structure keeping the multiplication fixed, while in \cite[Proposition~2]{Rum:set} one considers a new multiplication with addition unchanged. However, if $e$ is a left absorber in a unital truss $(T,[-,-,-],\cdot, 1)$, then using the translation heap automorphism \eqref{trans} one can induce a new associative product on $T$ by the formula:
$$
a*_eb = \tau_e^1\left(\tau_1^e(a)\cdot \tau_1^e(b)\right), \qquad \mbox{for all $a,b\in T$}.
$$
Then $(T,[-,-,-],*_e, [1,e,1])$ is a unital near-truss isomorphic to $(T,[-,-,-],\cdot, 1)$ in which $1$ is a left absorber. Consequently, $(T,+_1,*_e)$  is a unital near-ring corresponding to the skew-ring $(T,+_1,\cdot)$. In particular, if $B$ is a skew-ring with unit $u$, then $(T(B),+_1,*_{[1,u,1]}, u) = (B,+,*_{-u}, u)$ is the unital near-ring described in \cite[Proposition~2]{Rum:set}.
\end{remark}

\begin{example}\label{ex:skew-brace}
	Let $T$ be a near-truss such that $(T,\cdot)$ is a group with neutral element $1$. Then  $(T,+_{1},\cdot)$ is a skew brace. We call such $T$ a \textbf{brace-type near-truss}.

	Conversely, if $B$ is a skew brace, then one can assign to it the near-truss $(\hH(B),\cdot)$ which we will denote by $\tT(B)$. As was the case with the skew-rings, if we start with a skew brace $B$, assign to it the truss $\tT(B)$ and then take the retract at identity, we obtain the same skew brace.
\end{example}

Our goal is to describe the properties that a pre-truss $T$ and a congruence $\sim$ on it must have for the quotient near-truss $T/ \!\!\sim$ to be a  brace-type near-truss, i.e.\ a near-truss associated with a skew brace. The main theorem of this section is Theorem \ref{thm:main3} which states when a  near-truss $T/\!\!\sim$ can be associated with a  skew brace.  First we identify those normal sub-heaps of a pre-truss $T$ that faithfully correspond  to congruences.

\begin{definition}\label{def.act.paragon}
Let  $T$ be  a pre-truss. 
\begin{zlist}
\item A sub-heap $S$ of $T$ is said to be \textbf{left-closed} (resp.\ \textbf{right-closed})  if, for all $s,s'\in S$ and $t\in T$,
\begin{equation}\label{actions}
[ts',ts,s] \in S \qquad \mbox{(resp.\ $[s't,st,s] \in S$)}.
\end{equation}
\item A sub-heap $S$ that is left- and right-closed is said to be \textbf{closed}.
\item A non-empty normal sub-heap $P$ of $T$ such that every equivalence class of the sub-heap relation $\sim_P$ is a closed (normal) sub-heap of $T$ is called a \textbf{paragon}.
\end{zlist}
\end{definition}

Observe that Lemma~\ref{lem:map} implies that  if $P$ is a paragon in a pre-truss $T$, then all the equivalence classes of $\sim_P$ are mutually isomorphic paragons as well. 

\begin{remark}\label{rem.all.ex}
In the case of a non-empty sub-heap $S$ the quantifier `for all $s\in S$'  in the definition of the left or right closure property \eqref{actions} can be equivalently replaced by the existential quantifier. Indeed, assume that there exists $q\in S$ such that, for all $s'\in S$ and $t\in T$, $[ts',tq,q] \in S$. Then, for all $s\in S$,
$$
\begin{aligned}
{}[ts',ts,s] &= [[[ts',tq,q],q,tq],ts,s] = [[ts',tq,q], [ts,tq,q], s] \in S,
\end{aligned}
$$
by the associativity, Mal'cev's identities and \eqref{assoc+}, and since $S$ is a sub-heap. Similarly for the right closure property.
\end{remark}

\begin{lemma}\label{lem.par}
A non-empty normal sub-heap $P$ of a pre-truss $T$ is a paragon if and only if, for all $a,b\in T$ and $p,e\in P$,
$$
[a[p,e,b],ab,e] \in P\quad \&\quad [[p,e,b]a,ba,e] \in P.
$$
\end{lemma}
\begin{proof}
By Lemma~\ref{lem:map}, the equivalence class of $b\in T$  is  $\bar{b} = P_e^b = \{[p,e,b]\;|\; p\in P\}$, for all $e\in P$. Hence $\bar{b}$ is left-closed if and only if, for all $p\in P$ and $a\in T$, there exists $q\in P$ such that
$$
[a[p,e,b], ab, b]= [q,e,b], 
$$
that is, if and only if
$$
[a[p,e,b],ab,e] = q\in P,
$$
as required. 
\end{proof}

\begin{corollary}\label{cor.par}
A non-empty normal sub-heap $P$ of a near-truss $T$ is a paragon if and only if $P$ is left-closed and all equivalence classes of the induced sub-heap relation are right-closed. In particular $P$ is a paragon in a skew truss if and only if it is a closed normal sub-heap.
\end{corollary}
\begin{proof}
Since in a near-truss the left distributivity law holds, the left-closure  property in Lemma~\ref{lem.par}  reduces to $[ap,ae,e]\in P$, that is, the left-closedness of $P$. In a skew truss the right-closure property  is treated symmetrically.
\end{proof}

Corollary~\ref{cor.par} shows that, in the case of skew trusses (and hence trusses) the notion of a paragon introduced in Definition~\ref{def.act.paragon}  reduces to the notion introduced in \cite[Definition~3.15]{Brz:par}.

 \begin{lemma}\label{lem.ker.par}
 Let $f:T\lra T'$  be  a morphism of pre-trusses. 
 \begin{zlist}
 \item For all $z\in \im f$, $f^{-1}(z)$ is a paragon in $T$. In particular, if $P'$ is a paragon in $\im f$, then $f^{-1}(P')$ is a paragon in $T$. 
\item If  $P$ is a paragon  in $T$ then $f(P)$ is a paragon in $\im f$. 
\end{zlist}
 \end{lemma}
 \begin{proof}
 (1)  By \cite[Lemma~2.12]{Brz:par}, $f^{-1}(z)$ is a normal sub-heap which is non-empty (since $z\in \im f$). For all $a,b\in T$ and $p,e\in f^{-1}(z)$,
$$
\begin{aligned}
f([a[p,e,b],ab,e]) &= [f(a)[f(p),f(e),f(b)], f(a)f(b),f(e)]\\
& = [f(a)[z,z,f(b)], f(a)f(b),z] =z,
\end{aligned}
$$
since $f$ preserves multiplication and ternary operations, and by Mal'cev identities. Thus $[a[p,e,b],ab,e] \in f^{-1}(z)$. By the same arguments, $[[p,e,b]a,ba,e] \in f^{-1}(z)$.  In view of Lemma~\ref{lem.par} this means that $f^{-1}(z)$ is a paragon. 

Assume that $P'$ is a paragon. 
That the pre-image of a normal sub-heap is a normal sub-heap follows by the standard group-theoretic arguments.  
Since $f$ preserves multiplication and heap operation, for all $a,b\in T$ and $p,q \in f^{-1}(P')$,
$$
\begin{aligned}
f\left([a[p,q,b],ab,q]\right) &= \left[f(a)[f(p),f(q),f(b)],f(a)f(b),f(q)\right] \quad \&\\
f\left([[p,q,b]a,ba,q]\right) &= \left[[f(p),f(q),f(b)]f(a),f(b)f(a),f(q)\right].
\end{aligned}
$$
Since $P'$ is a paragon, and $f(p), f(q)\in P'$, both expressions are elements of $P'$. Therefore, $[a[p,q,b],ab,q],[[p,q,b]a,ba,q] \in f^{-1}(P')$, and hence $ f^{-1}(P')$ is a paragon. 

Statement (2) is proven by similar arguments.
 \end{proof}

\begin{theorem}\label{thm.par}
Let $P$ be a non-empty normal sub-heap of a pre-truss $T$. Then  the canonical heap map $\pi:T\to T/P$ is a homomorphism of pre-trusses if and only if $P$ is a paragon.
\end{theorem}
\begin{proof}
Assume that $\pi$ is a pre-truss homomorphism. Since $P=\pi^{-1}(P)$, $P$ is a paragon by Lemma~\ref{lem.ker.par}.

For the proof of the opposite implication assume that $P$ is a paragon. Then $\sim_{P}$ is a congruence on the heap $T$, so we only need to show that this relation is a congruence on the pre-truss $T$ as well. Let $a,b\in T$ be such that $a\sim_{P}b$, so that $a,b\in \pi(b)$. 
Since $P$ is a paragon, for all $t\in T$, $[ta,tb,b]\in \pi(b)$.
Hence, $[\pi(tb),\pi(ta),\pi(b)] = \pi(b)$, that is, $\pi(tb)=\pi(ta)$ or, equivalently,  $ta\sim_{P}tb$.
In the same way one can prove that $a\sim_{P}b$ implies $at\sim_{P}bt$ for all $t\in T$. Assume that $a\sim_{P}b$ and $c\sim_{P}d$. Then $ac\sim_{P}bc$, $bc\sim_{P}bd$ and $ac\sim_{P} bd$, since $\sim_{P}$ is an equivalence relation. Therefore, $\sim_{P}$ is a congruence and the canonical map $\pi: T\to T/P$ is a homomorphism of pre-trusses. This completes the proof.
\end{proof}

\begin{corollary}\label{cor:nrec}
Let $N$ be a near-ring. Then  $P\subseteq N$ is an equivalence class for a congruence on $N$ if and only if $P$ is a paragon in $\tT(N)$  
\end{corollary}
\begin{proof}
Let us assume that $P$ is an equivalence class for a congruence on $N$, let $\bar{N}$ be the quotient near-truss with canonical homomorphism $\pi:N\lra \bar{N}$. Since $\pi$ is also a homomorphism of associated near-trusses, that is, $\pi:\tT(N)\lra \tT(\bar{N})$, and $P=\pi^{-1}(P)$,  $P$ is a paragon in $\tT(N)$ by Lemma~\ref{lem.ker.par}.

In the opposite direction, assume that $P$ is a paragon in $\tT(N)$. Then there exists a near-truss homomorphism $\pi :\tT(N)\to \tT(N)/P$. Observe that the triple $(\tT(N)/P,+_{\pi(e)},\cdot)$, where $e$ is the neutral element of $N$, is a near-ring, since the image of a left absorber through a near-truss homomorphism is a left absorber. Therefore $\pi$ is also a homomorphism of  the retracted near-rings and $P$ is an equivalence class of a congruence given by $\pi$ as $P=\pi^{-1}(P)$.
\end{proof}

\begin{lemma}\label{lem.brace.par}
Let $\tT(B)$ be a near-truss associated to a skew brace $B$ (with identity 1). Then $P$ is a paragon in  $\tT(B)$ if and only if, for all $p\in P$, $P_p^1$ is an ideal in $B$. 
\end{lemma}
\begin{proof}
Assume that $P$ is a paragon in $\tT(B)$. 
Then $1\in P_{p}^{1}$,  $(P_p^{1},+_{1})$ is a normal subgroup of $(B,+)$ as $P_{p}^{1}$ is a normal sub-heap and $+_1=+$. Since $P_{p}^{1}$ is closed, for all $a\in B$ and $b\in P_{p}^{1}$, 
 $$
 ab - a = [ab,a1,1] \in P_{p}^{1} \quad \&\quad ba-a= [ba,1a,1] \in P_{p}^{1}.
 $$
 Therefore, $ba-ab=c \in  P_{p}^{1}$, and, using the brace distributive law,
 $$
 a^{-1}ba = a^{-1}(c+ab)= a^{-1}c - a^{-1} +b \in P_{p}^{1},
 $$
 since $P_{p}^{1}$ is left-closed. This implies that $a^{-1}P_p^1a=P_p^1$, that is, $aP_p^1=P_p^1a$, and completes the proof that  $P_{p}^{1}$ is an ideal in $B$. 
 
 Conversely, if $P_{p}^{1}$ is an ideal in $B$, then $B/ P_{p}^{1}$ is a brace by \cite[Lemma~2.3]{GuaVen:ske}, and the canonical brace epimorphism $\pi: B\lra B/ P_{p}^{1}$ induces a near-truss morphism $\pi: \tT(B) \lra \tT(B/ P_{p}^{1})$. Since $P_{p}^{1} = \pi^{-1}(P_{p}^{1})$, $P_{p}^{1}$ and consequently also  $P = \left(P_{p}^{1}\right)_1^p$ are paragons by Lemma~\ref{lem.ker.par}.
\end{proof}

\begin{corollary}\label{cor.brace.par}
Let $B$ be a skew brace, then $P\subseteq B$ is an equivalence class for some congruence on $B$ if and only if $P$ is a paragon in $\tT(B)$.  
\end{corollary}
\begin{proof}
 The proof of the left to right implication is the same as in Corollary~\ref{cor:nrec}. The other implication follows by Lemma~\ref{lem.brace.par}.
 \end{proof}
 
To connect quotients of near-trusses with skew braces we need to determine which paragons do not produce absorbers in the quotients. To this end we introduce the notion of an ideal.

\begin{definition}
A normal sub-heap $I$ of a pre-truss $T$ is called  a \textbf{left} (resp.\ \textbf{right}) \textbf{ideal} if, for all $t\in T$ and $i\in I$, $ti\in I$ (resp. $it\in I$). If $I$  is both left and right ideal, then it is called an \textbf{ideal}. A left (resp.\ right) ideal is said to be \textbf{maximal} if it is not contained in any left (resp.\ right) proper ideal.
\end{definition}

Note that an ideal is a closed sub-heap, but this does not yet make it into a paragon, since the equivalence classes of the corresponding sub-heap relations need not be closed. Also note that if $f:T\to T'$ is a homomorphism of pre-trusses, then the pre-image of an ideal in $\im f$ is an ideal in $T$ and the image of an ideal in $T$ is an ideal in $\im f$.

\begin{lemma}\label{lem:par-cont}
If a left-closed normal sub-heap of a pre-truss contains a left ideal, then it is a left  ideal. 
\end{lemma}
\begin{proof}
Let $P$ be a left-closed normal sub-heap of $T$, and let $I$ be a left ideal such that $I\subseteq P$. Then, for all $p\in P$, $t\in T$ and $i\in I$,
$
tp=[[tp,ti,i],i,ti]\in P$, since $[tp,ti,i]\in P$ and $
ti,i\in I\subseteq P$.
\end{proof}

\begin{lemma}\label{lem:abs}
Let $T$ be a pre-truss and $P$ be a paragon. Then  $T/P$ has a left absorber if and only if there exist $a\in P$ and $t\in T$ such that $P_{a}^{t}$ is a left ideal. 
\end{lemma}
\begin{proof}
The assertion follows from the fact that for every $a\in P$ and $t\in T$, $P_{a}^{t}=\pi(t)$, where $\pi$ is the canonical surjection onto the quotient $T/P$. 
\end{proof}

\begin{corollary}
If $I$ is a paragon that is a right ideal in a pre-truss $T$, then for all $e\in T\setminus I$ and all $a\in I$, $I_{a}^{e}$ is not a left ideal.
\end{corollary}
\begin{proof}
We know from Lemma~\ref{lem:map} that $T/I = T/I_{a}^{e}$.  Assume that $I$ is a right ideal and suppose that $I_{a}^{e}$ is a left ideal. Then, by Lemma~\ref{lem:abs}, $I$ is a right absorber in $T/I$ and $I_{a}^{e}$  is a left absorber in $T/I_{a}^{e}$. Hence $I=I_{a}^{e}$. But $e\not\in I$ and $e\in I_{a}^{e}$, which yields a required contradiction and completes the proof.
\end{proof}

\begin{proposition}\label{lem:sim-rb}
Let $T$ be a unital near-truss. 
\begin{zlist}
\item  $T$ is a truss associated with a skew brace if and only if  $T$ has exactly one left ideal. 
\item  $T$ is a truss associated with a near-field if and only if $T$ has a left absorber and exactly two left ideals. 
\end{zlist}
\end{proposition}
\begin{proof}

(1) Assume that $T$ has exactly one left ideal. 
For all $x \in T$ the left ideal $Tx:=\{tx\;|\; t\in T\}$ has to be the whole of $T$ (in particular if $T$ has at least two elements, then it has no left absorbers). Therefore, there exists $y\in T$ such that $yx=1$ and $y$ is a left inverse to $x$. As $x$ is an arbitrary element there exists $x'$ such that $x'y=1$. Thus $(x'y)x=x$ and by associativity $x'=x$. The conclusion is that $y$ is the two-sided inverse of $x$ and the monoid $(T, \cdot)$ is a group. Therefore, the near-truss $T$ is a brace-type near-truss (see \cite[Corollary 3.10]{Brz:par}).

Conversely, suppose that $T=\tT(B)$ for a skew brace $B$ and that there exists a left ideal $I\subsetneq \tT(B).$ Observe that if $x\in I$, then $x^{-1}x=1\in I$, therefore $I=T$. This contradicts the assumption that $I\neq T$. Thus $T$ has exactly one left ideal. 

(2) Let us assume that $T$ has a left absorber and exactly two left ideals. Then there exists a near-ring $R$ such that $T=\tT(R)$, to be precise $R$ is the retract $(\tT(R),+_e)$, where $e$ is the left absorber. Seeking contradiction, suppose that $R$ is not a near-field. Then there exists a left ideal $\{e\}\neq I\subsetneq R$; but $I$ is also a left ideal of $\tT(R)$, which contradicts with the assumption that $T$ has only two left ideals. 
Therefore, $R$ is a left near-field.

Assume that $T=\tT(F)$, where $F$ is a left near-field, then $0$ (the neutral element for the addition in $F$) is a left absorber in $T$.  Suppose by contradiction that $\tT(F)$ has a left ideal  $\{0\}\neq I\subsetneq \tT(F)$. Consider, for any $a\in I$ the ideal $I_{a}^{0}:=\{[b,a,0]\;|\; b\in I\}$. The ideal $I_{a}^{0}$ is neither equal to $\{0\}$ nor to $T$, since the map $[-,a,0]$ is a bijection. Furthermore, $I_{a}^{0}$ is an ideal in $F$, and hence $F$ is not a near-field. This contradicts with the assumption that $F$ is a near-field. 
\end{proof}

\begin{lemma}\label{lem:max-id-sim}
Let $T$ be a near-truss. If $I$ is a paragon in $T$ that is a left maximal  ideal, then $T/I$ has no  ideals different from a singleton set and $T/I$.
\end{lemma}
\begin{proof}
Suppose that $\mathfrak{J}\neq T/I$ is a left ideal in $T/I$ that is not a singleton set. Since $I$ is a left absorber in $T/I$, 
for any element $J\in \mathfrak{J}$, $\mathfrak{J}_{J}^{I}$ is a left ideal in $T/I$ by the left distributive law. Hence, $\pi^{-1}(\mathfrak{J}_{J}^{I})$ is a left ideal in $T$, where $\pi:T\lra T/I$ is the canonical surjection. Moreover, $I\subset \pi^{-1}(\mathfrak{J}_{J}^{I})$, since $I\in \mathfrak{J}_{J}^{I}$. Therefore, since $I$ is left maximal, either $I= \pi^{-1}(\mathfrak{J}_{J}^{I})$, and hence 
$\mathfrak{J}_{J}^{I} =\{I\}$, which implies that $\mathfrak{J}=\{J\}$, or $\pi^{-1}(\mathfrak{J}_{J}^{I}) = T$, which implies in turn that $\mathfrak{J} = T/I$.  Thus both cases lead to a contradiction.
\end{proof}

Although dividing by a paragon which is a left maximal ideal  yields a near-truss without proper left ideals, this near-truss always has an absorber. Therefore it is never a brace-type near-truss. The most straightforward idea to generalise maximality to paragons leads us to the following definition:

\begin{definition}
Let $T$ be a pre-truss. A left-closed (resp.\ right-closed) normal  sub-heap $P\subsetneq T$ is said to be \textbf{maximal} if it is not contained in any left-closed (resp.\ right-closed) sub-heap other than $T$. A paragon $P$ is said to be \textbf{left maximal} (resp.\ \textbf{right-maximal}, \textbf{maximal}) if it is a maximal left-closed (resp.\ right-closed, left- and right-closed) sub-heap. 
\end{definition}

\begin{lemma}
Let $T$ be a near-truss or a skew-truss and $P$ be a left-closed normal sub-heap. Then $P$ is maximal if and only if, for all $a\in P$ and $t\in T$, $P_{a}^{t}$ is a maximal left-closed normal sub-heap.
\end{lemma}
\begin{proof}
Note that by the normality of $P$ and the left distributive law, all the $P_{a}^{t}$  are left-closed normal sub-heaps. Seeking contradiction assume that $P$ is maximal and there exists $a\in P$ and $t\in T$ such that $P_{a}^{t}$ is not maximal. Then there exists a left-closed normal sub-heap $Q$ such that $P_{a}^{t}\subsetneq Q\subsetneq T$. Since $\tau_a^t$ is an isomorphism with the inverse $\tau_t^a$, this implies that $P\subsetneq Q_{t}^{a}\subsetneq T$. Hence $P$ is not maximal, contrary to the assumption. 

The opposite implication is also easily deduced from the fact that $P=(P_{a}^{t})_{t}^{a}.$
\end{proof}

\begin{remark}
In the case of rings the notion of maximal ideals and maximal paragons coincide as every paragon $P$ in the ring can be associated with an ideal $P_{a}^{0}$ for any $a\in P$ and an absorber $0$.
\end{remark}

\begin{lemma}\label{lem:par-id}
Let $T$ be a near-truss or a skew-truss and $P \subseteq T$ a left maximal paragon, then $T/P$ has no proper (i.e.\ different from singletons and the whole of $T/P$) left ideals.
\end{lemma}
\begin{proof}
By the definition of maximality of $P$, $T/P$ has no proper left paragons. Therefore it has no proper left ideals as a left ideal is a left paragon.
\end{proof}

Observe that by dividing a near-truss without left absorbers by a paragon which is left-maximal one obtains a near-truss associated with a skew brace. If the quotient is a skew brace, then it is a simple brace, that is, it has no ideals in the sense of sub-braces different from the skew brace itself and singleton subsets of it. Maximal paragons do not characterise all the quotients which are  brace-type near-trusses, since there exist skew braces that are non-simple.

\begin{theorem}\label{thm:main3}
Let $T$ be a unital near-truss and $P$ be a paragon,  and let $\pi_P:T\lra T/P$ be the canonical epimorphism. Then $T/P$ is a brace-type near-truss if and only if, for all left ideals $I\subsetneq T$ and $\|a\in T/P$, $\pi_P^{-1}(\|a)\not\subseteq I$.
\end{theorem}
\begin{proof}
Let us assume that $T/P$ is a  brace-type near-truss. Observe that should $\pi_{P}^{-1}(\|a)\subseteq I$ for a left ideal $I$, then $\pi_{P}(I)$ would be a left ideal in $T/P$. Thus, $\pi_{P}(I)=T/P$, since  $T/P$ is a brace-type near-truss. On the other hand, if $c\in T\setminus I$ then $\pi_P(c)\not\in \pi_P(I)$. Indeed, should $\pi_P(c)\in \pi_P(I)$, then there would exist $i\in I$ and $p\in P$ such that $[c,i,p]\in P$. Thus, for all $a\in \pi_P^{-1}(\|a)$,  $[c,i,a]=[[c,i,p],p,a]\in \pi_P^{-1}(\|a)\subset I$ and $c\in I$. Therefore, $I=T$.

Now, assume that, for all left ideals $I\subsetneq T$ and $\|a\in T/P$, $\pi_{P}^{-1}(\|a)\not\subseteq I$ and $T/P$ is not a  brace-type near-truss. Then there exists a left ideal $\mathfrak{J}\subsetneq T/P$. The pre-image $\pi_{P}^{-1}(\mathfrak{J}) \subsetneq T$ is a left ideal in $T$ and, obviously, for any $\|j\in\mathfrak{J}$, $\pi_{P}^{-1}(\|j)\subseteq \pi_{P}^{-1}(\mathfrak{J})$. This contradicts the assumption that, for all $\|a\in T/P$, $\pi_{P}^{-1}(\|a)\not\subseteq \mathfrak{J}$, so $T/P$ is a brace-type near-truss. The proof is completed.
\end{proof}

\begin{example}\label{ex:3.1}
Let $B$ be a skew brace and $R$ a ring. One can consider the product near-truss $\tT(B)\times \tT(R)$ with operation given by $(b,r)(b',r')=(bb',rr')$, for all $(b,r),(b',r')\in \tT(B)\times \tT(R)$. It is easy to check that, for any ideal $I$ in $R$, $\tT(B)\times I$ is an ideal in $\tT(B)\times \tT(R)$ and that for any paragon $P$ in $\tT(B),$ $P\times I$ is a paragon in $\tT(B)\times \tT(R)$. Every paragon of the form $P\times \tT(R)$ fulfills conditions in Theorem \ref{thm:main3} and one easily finds that $(\tT(B)\times \tT(R))/(P\times \tT(R))\cong \tT(B)/P$
\end{example}

\begin{example}\label{ex:3.2}
Let $T = 2\ZZ + 1$. The set $P = \{2^nm + 1 \text{ \ | \  } m \in T\} \subset T$  is a paragon and the quotient $T/P$ is a brace-type truss isomorphic to  $U(\ZZ/2^{n+1}\ZZ)$, the sub-truss of all units in the quotient ring $\ZZ/2^{n+1}\ZZ$. To prove that this isomorphism holds it is first of all helpful to notice that $|T/P| = 2^n = |U(\ZZ/2^{n+1}\ZZ)| $. Indeed, there are as many classes in the quotient as the odd numbers between $2^nm+1$ and $2^n(m+2) + 1$ (it is important to notice that, if $m$ is odd, then $m+1$ is even), so exactly $2^n$. Then the isomorphism is given by sending $\overline{2m+1} \in T/P$ to $2m+1 \textup{ mod } 2^{n+1}$: this is evidently injective, so also surjective since the two sets have the same size, and it is easily proven to be a homomorphism.
\end{example}

\section{Domains and completely prime paragons}\label{sec:4}
The aim of this section is to introduce the notion of a completely prime paragon. This, in analogy to the case of rings, should lead to a quotient pre-truss that is a domain, i.e.\ a pre-truss in which cancellation properties hold. After describing such paragons, the next step is to consider the Ore localisation for pre-trusses, which is the subject of the following section. By inverting all elements of a domain we should obtain a pre-truss without proper left ideals and with no absorbers, so if the distributive law holds this will be a near-truss associated with a skew brace. Let us start with the definition of a domain. When working with rings, there is always an absorber which in many cases allows for simplification of some conditions. Not all pre-trusses have an absorber (in fact, having brace applications in mind, we are particularly interested in those that do not have absorbers), so many of the well-known definitions need to be in some sense generalised or stated without involving any absorber. We begin with the definition of a regular element:
\begin{definition}\label{def:reg.el}
Let $T$ be a pre-truss. An element $a\in T^\Abs$ is said to be \textbf{left regular} (resp.\ \textbf{right regular}) if, for all $b\not=c$,
\begin{equation}\label{ue:1}
 ab\not= ac\ \qquad \mbox{(resp.\ $ba\not=ca$)}.
\end{equation}
If $a$ is both left and right regular element then it is said to be \textbf{regular}.
\end{definition}
Observe that conditions $\eqref{ue:1}$ can be written in a way that makes them reminiscent of the closedness conditions \eqref{actions} used in the definition of a paragon.
The statement that $ac\not=ab$ is equivalent to saying that $[ac,ab,b]\not=b$. Similarly, $ba\not=ca$ is equivalent to say that $[ca,ba,b]\not= b$. This indicates that  these conditions  are closely related to the definition of paragon.
\begin{lemma}\label{lem:unitoexist}
Let $T$ be a near-truss. Then $a\in T$ is a left regular element if and only if there exists an element $c$ such that, for all $b\in T\setminus\{c\},$
\begin{equation}\label{ue:2}
ab\not= ac.
\end{equation}
\end{lemma}
\begin{proof}
If $a$ is left regular then, for all $c\in T$ and all $b\in T\setminus\{c\}$, the inequality \eqref{ue:2} holds, which implies the existence of $c$. 

Assume that there exists $c\in T$ such that, for all $b\not=c$, $ab\not=ac$. Thus $[ab,ac,ac']=a[b,c,c']\not=ac'$, for all $c'\in T$. Note that, for all $c,c'\in T$, the map 
$$
[-,c,c']:T\setminus\{c\}\lra T\setminus\{c'\}, \qquad b\lto [b,c,c'],
$$ 
is a bijection. Therefore, for all $t\in T\setminus\{c'\}$, $at\not=ac'$. By the arbitrariness of $c'$, $a$ is a left regular element. This completes the proof.
\end{proof}

\begin{lemma}\label{lem.reg.ring}
Let $R$ be a ring. Then $a\in R$ is a regular element if and only if $a$ is a regular element in $\tT(R).$
\end{lemma}
\begin{proof}
The equivalence will be proven for left regularity only, the right regularity case in symmetric.
Let us assume that $a\in R$ is a regular element. Then there is no $b\in R\setminus \{0\}$ such that $ab=0$. Thus, by Lemma~\ref{lem:unitoexist}, if $c=0$ in \eqref{ue:2}, then $a$ is a regular element in $\tT(R)$, since $a$ is regular in $R$. 

Suppose that $a$ is regular in $\tT(R).$ Then $ab\not=ac$ implies $a(b-c)\not=0$. Therefore, by substituting  $b=t+c$, $at\not= 0$ for all $t\in R\setminus\{0\}$, which completes the proof. 
\end{proof}

Now we are ready to introduce the definition of a domain in clear analogy with the usual notion for rings.

\begin{definition}\label{def:domain}
	A pre-truss $T$ is called a domain if all elements of $T^\Abs$ are regular.
\end{definition}

In view of Lemma~\ref{lem.reg.ring},
a ring $R$ is a domain if and only if $\tT(R)$ is a domain.

\begin{lemma}\label{lem:cancel}
	 A near-truss $T$ is a domain if and only if it satisfies the cancellation property, that is for all $a\in T^{\Abs}$ and $b,b'\in T$,  each one of the equalities $ab=ab'$ or $ba=b'a$ implies that $b=b'$.
\end{lemma}
\begin{proof}
This follows immediately for the definitions of a regular element and a domain.
\end{proof}

\begin{definition}\label{def:prime}
Let $T$ be a pre-truss. A non-empty paragon $P\subseteq T$ is said to be  \textbf{completely prime} if, for all $p\in P$, $a,b,c\in T$,
$$
[ab,ac,p]\in P\implies P_{p}^{a} \text{\  is\  an\  ideal\  or\  } [b,c,p]\in P
$$
 \center{and}
 $$
  [ba,ca,p]\in P\implies P_{p}^{a} \text{\  is\  an\  ideal\  or\  }[b,c,p]\in P.
$$
\end{definition}

\begin{lemma}
Let $T$ be a pre-truss and $P$ be a non-empty paragon. Then $P$ is completely prime if and only if, for all $p\in P$ and $t\in T$, $P_{p}^{t}$ is completely prime.
\end{lemma}
\begin{proof}
Let us assume that $P$ is a completely prime paragon and let $p\in P$ and $t\in T$. We know that $P_{p}^{t}$ is a paragon (see comment that follows Definition~\ref{def.act.paragon}). Then, for all $a,b,c\in T$ and $q\in P$,
$[ab,ac,[q,p,t]]\in P_{p}^{t}$ implies $[[ab,ac,[q,p,t]],t,p]=[ab,ac,q]\in P$, since $(P_{p}^{t})_{t}^{p}=P$. Thus, $P_{q}^{a}$ is an ideal or $[b,c,q]\in P$. In view of $(P_p^t )_{[q,p,t]}^a = P_{q}^{a}$, the first option is equivalent to $(P_p^t )_{[q,p,t]}^a$ being an ideal and the second to $[b,c,[q,p,t]]\in P_{p}^{t}$. Hence $P_{p}^{t}$ fulfils the left condition to be a completely prime paragon. Analogously one can prove that $P_{p}^{t}$ satisfies the right condition. Therefore, $P_{p}^{t}$ is a completely prime paragon. 
\end{proof}

Unsurprisingly, the distributive laws yield simplification of the definition of a completely prime paragon.
\begin{lemma}\label{lem.prime.short}
Let $T$ be a  skew truss and $P$ be a paragon. Then $P$ is completely prime if and only if 
there exists $p \in P$ such that, for all $a,d\in T$, \
$$
[ad,ap,p]\in P\implies P_{p}^{a} \text{\  is\  an\  ideal\  or\  } d\in P 
$$
\begin{center}
and
\end{center}
$$
[da,pa,p]\in P\implies  P_{p}^{a} \text{\  is\  an\  ideal\  or\  } d\in P.
$$
\end{lemma}
\begin{proof}
It is sufficient to observe that, for every $b\in T$, $[b,c,p]$ can be substituted by some $d\in T$ since $[-,c,p]:T\lra T$ is a bijection with the inverse given by $[-,p,c]:T\lra T$. Hence, if $b=[d,p,c]$, $d=[[d,p,c],c,p]$, and so
$$
[ab,ac,p] = [a[d,p,c],ac,p]= [ad,ap,p] \;\; \& \;\; [ba,ca,p] = [[d,p,c]a,ca,p]= [da,pa,p],
$$
by the distributive laws and the axioms of a heap. This completes the proof.
\end{proof}

\begin{lemma}
If $P\subsetneq T$ is a completely prime paragon in a pre-truss $T$, then, for all $p \in P$ and for all left (right) absorbers  $a,a'\in T$, $P_{p}^{a}=P_{p}^{a'}$. 
\end{lemma}
\begin{proof}
Let $a$ be a left absorber. For all $b,c\in T$ and $p\in P$, $[ba,ca,p]=[a,a,p]=p\in P$, so $P_{p}^{a}$ is an ideal or $[b,c,p]\in P$. The second option is equivalent to $b\sim_{P} c$, for all $b,c\in T$. Observe, though, that since $P\neq T$, there exist $b,c\in T$ such that $b\not\sim_{P} c$. Therefore, $P_{p}^{a}$ is an ideal and $\|{a}\in T/P$ is an absorber. From the fact that if a truss has an absorber then it has only one left absorber one concludes that $P_{p}^{a}=P_{p}^{a'}$, for all left absorbers $a,a'$.
\end{proof}

\begin{theorem}\label{thm:dom-prime}
Let $T$ be a pre-truss. Then $P$ is a completely prime paragon if and only if $T/P$ is a domain.
\end{theorem}
\begin{proof}
We write $\|a$ for the class  of $a$ in $T/P$. The pre-truss $T/P$ is a domain if and only if, for all $\|a,\|b,\|c\in T/P$, $\|{ab}=\|{ac}$ implies that $\|b=\|c$ or $\|a$ is an absorber. The equality $\|{ab}=\|{ac}$ amounts to the   existence of $p\in P$ such that $[ab,ac,p]\in P$. Observe that $\|b=\|c$ if and only if $[b,c,p]\in P$, and $\|a$ is an absorber if and only if $P_{p}^{a}$ is an ideal. 
The proof proceeds analogously for the right cancellation property.
\end{proof}

\begin{remark}
Every paragon in a near-truss $\tT(B)$ associated with a skew brace $B$ is completely prime.
\end{remark}

\begin{corollary}
Let $R$ be a ring. An ideal $I$ is completely prime in $R$ if and only if $I$ is a completely prime paragon in $\tT(R)$.
\end{corollary}
\begin{proof}
Let us assume that $I$ is a completely prime ideal in $R$. Then, for all $a,b\in R$ and absorber $0\in I$,
$$
[ab,a0,0]=ab\in I\implies a\in I\; \text{or}\; b\in I.
$$
Thus, if $a\in I$, then $I_{0}^{a}=I$ is an ideal, and hence $I$ is a completely prime ideal in $\tT(R)$.

Conversely, assume that $I$ is a completely prime paragon in $\tT(R)$. For all $a,b\in \tT(R)$,
$$
ab=[ab,a0,0] \in I \implies I_{0}^{a}\text{\  is\  an\  ideal\  or\  } b\in I.
$$
Observe that $I_{0}^{a}$ is an ideal if and only if $a\in I$. Therefore, $I$ is a completely prime ideal in $R$. This completes the proof.
\end{proof}

\begin{lemma}
Let  $f:T\to T'$ be a morphism of pre-trusses. If $P$ is a completely prime paragon in $\im f$, then $f^{-1}(P)$ is a completely prime paragon in $T$.
\end{lemma}
\begin{proof}
By Lemma~\ref{lem.ker.par}, $f^{-1}(P)$ is a paragon. For all $a,b,c\in T$ and $p\in f^{-1}(P)$, if $[ab,ac,p]\in f^{-1}(P)$, then
$$f([ab,ac,p])=[f(a)f(b),f(a)f(c),f(p)]\in P.
$$
This implies that $P_{f(p)}^{f(a)}$ is an ideal or  $f([b,c,p])=[f(b),f(c),f(p)]\in P$. Therefore, $[b,c,p]\in f^{-1}(P)$ or $P_{f(p)}^{f(a)}$ is an ideal. Let us assume that 
\[
z\in f^{-1}\left ( P_{f(p)}^{f(a)} \right ) =\{x\in T\;|\;\exists{q\in P}\;\; \text{ s.t. }\;\; f(z)=[q,f(p),f(a)]\} .
\]
 Then $f(z)=[q,f(p),f(a)]$, for some $q\in P$ and $f([z,a,p])=[f(z),f(a),f(p)]=q\in P$. Hence $z=[[z,a,p],p,a]\in f^{-1}(P)_{p}^{a}$ and $f^{-1}(P_{f(p)}^{f(a)})\subseteq f^{-1}(P)_{p}^{a}$. Therefore, $f^{-1}(P)_{p}^{a}\subseteq f^{-1}(P_{f(p)}^{f(a)})$ and by Lemma~\ref{lem:par-cont}, $f^{-1}(P)_{p}^{a}$ is an ideal. This completes the proof.
\end{proof}

We conclude this section with an example of a completely prime paragon and the corresponding quotient domain.

\begin{example}\label{ex.poly}
Let $O(x)$ be the set of all polynomials in $\ZZ[x]$ in which the sum of coefficients is odd. One can easily check that $O(x)$ is a sub-monoid of the multiplicative monoid $\ZZ[x]$ and a sub-heap of $\ZZ[x]$ with the standard operation 
$[p,q,r] = p-q+r$. All this means that $O(x)$ is a (commutative) truss. 

Take any $t_0,t_1\in O(x)$ and define 
$$
P(t_0,t_1) := \{ p\in O(x)\; |\; (t_1-t_0) \, \mbox{divides}\, (p-t_0)\}.
$$
Then $P(t_0,t_1)$ is a paragon in $O(x)$ and it is a completely prime paragon provided that $t_1-t_0$ is irreducible in $\ZZ[x]$.
\end{example}
\begin{proof}
Clearly, if $p-t_0$, $q-t_0$ and $r-t_0$ are divisible by $t_1-t_0$, then so is $[p,q,r]-t_0 =p-q+r -t_0$. Hence $P(t_0,t_1)$ is a sub-heap of $O(x)$. Note that $t_0\in P(t_0,t_1)$, and hence, for all $p\in P(t_0,t_1)$ and  $q\in O(x)$,
$$
[qp,qt_0,t_0] -t_0 = qp-qt_0 = q(p-t_0).
$$
Therefore, $[qp,qt_0,t_0] = [pq,t_0q,t_0]\in P(t_0,t_1)$, which means that $P(t_0,t_1)$  is a paragon. 

Now assume that $c=t_1-t_0$ is irreducible in $\ZZ[x]$, and take $a,b\in O(x)$ for which there exists $p\in P(t_0,t_1)$ such that $[ab,ap,p]\in P$, that is $c\mid a(b-t_0)$. Since $c$ is irreducible, then either $c \mid (b-t_0)$, in which case $b\in P$, or $c\mid a$, that is, there exists $q\in \ZZ[x]$ such that $a=cq$. In this case,
$$
P(t_0,t_1)_{p}^a = \{r-p+cq\;|\; r\in P(t_0,t_1)\}.
$$
Thus $P_{p}^a$ contains all elements of $O(x)$ divisible by $c$ (since $c\mid (r-p)$, for all $r,p\in P(t_0,t_1)$), and hence it is an ideal in $O(x)$. Combined with the commutativity of $O(x)$, Lemma~\ref{lem.prime.short} yields that $P(t_0,t_1)$ is a completely prime paragon.
\end{proof}

Note that in general in the situation described in Example~\ref{ex.poly},
$$
\|a =\|b \in O(x)/P(t_0,t_1) \quad \mbox{if and only if} \quad (t_1-t_0)\mid (a-b).
$$
So, for example, take $t_0 = x$ and $t_1 = x^2+x+1$. Then $c = t_1-t_0 = x^2 +1$ is an irreducible polynomial in $\ZZ[x]$ and $O(x)/P(x,x^2+x+1)$ is a domain that can be identified with the sub-truss $O(i)$ of the truss (ring) of Gaussian integers $\ZZ[i]$,  defined as 
$$
O(i) = \{m+ni \;|\; \mbox{$m+n$ is odd}\}.
$$

\section{Skew braces of fractions}\label{sec:5}
To summarise, up to now we have introduced the notions of a domain and  a completely prime paragon, so that as long as we start with a pre-truss that has a completely prime paragon we can quotient out by it and obtain a domain. The next, and most important step, is to introduce localisation for pre-trusses. As the main goal of this section is to produce braces from near-trusses we will consider near-trusses without left absorbers and we will focus on localisation in the entire near-truss (to construct a ``brace of fractions") following Ore's classic construction \cite{Ore:Ore}.  First observe that since not every ring can be localised the same is true for trusses. Following \cite{Ore:Ore} we start by defining a regular pre-truss.

\begin{definition}\label{def:reg}
A pre-truss $T$ is said to be \textbf{left regular} if $T$ is a domain and it satisfies the left 
Ore condition, that is,
  for all $x,y\in T^{\Abs}$,  there exist $r,s\in T^{\Abs}$ such that $rx=sy$. 
  \end{definition}
In other words, a pre-truss is 
left (resp.\ right) regular if and only if $T^{\Abs}$ is a left 
 Ore set.
Next, we define the fraction relation on $T^{\Abs}\times T$, by
$$ 
(b,a)\sim (b',a') \;\;\; \mbox{if and only if there exist $\beta,\beta'\in T^{\Abs}$, s.t.\ $\beta b=\beta'b'$ \text{and}  $\beta a = \beta' a'$}.
$$ 
This is an equivalence relation by the same arguments as in \cite[Section 2]{Ore:Ore}. The equivalence class of $(b,a)$ is denoted by  $\frac{a}{b}$ and called a \textbf{fraction}, and the quotient set  $T^{\Abs}\times T/\sim$ is denoted by $\qQ(T)$.

\begin{theorem}\label{thm:loc}(Ore localisation for regular pre-trusses) Let $T$ be a left regular pre-truss. Then  $\qQ(T)$
is a pre-truss with the following operations
\begin{blist}
\item
For all $\frac{a}{b},\frac{a'}{b'},\frac{a''}{b''}\in \qQ(T)$, the ternary operation is defined by
\begin{equation}\label{frac.heap}
\left[\frac{a}{b},\frac{a'}{b'},\frac{a''}{b''}\right]:=\frac{[\beta_1 a,\beta_2a',\beta_3a'']}{\beta_1 b}=\frac{[\beta_1 a,\beta_2 a',\beta_3a'']}{\beta_2b'}=\frac{[\beta_1 a,\beta_2a',\beta_3a'']}{\beta_3b''},
\end{equation}
where $\beta_1, \beta_2, \beta_3 $ are any elements of $T^\Abs$ such that  $ \beta_1 b=\beta_2 b'= \beta_3 b'' $. 
\item For all $\frac{a}{b},\frac{a'}{b'}\in \qQ(T)$, 
\begin{equation}\label{frac.mon}
\frac{a}{b}\cdot\frac{a'}{b'}:=\frac{\gamma a'}{\gamma' b},
\end{equation}
where $\gamma,\gamma' \in T^\Abs$ are such that $\gamma b'=\gamma' a$.
\end{blist}
Furthermore, $(\qQ(T)^{\Abs},\cdot)$ is a group. We will call  $\qQ(T)$ the \textbf{pre-truss of (left) fractions} of $T$.
\end{theorem}
\begin{proof}
We follow closely the proof of \cite[Theorem 1]{Ore:Ore}. 
The multiplication of fractions \eqref{frac.mon} is defined in such a way that $\frac{a}{b}$ can be interpreted as $b^{-1}a$. Since it relies entirely on the properties of the semigroup $(T,\cdot)$, the arguments of the proof of \cite[Theorem 1]{Ore:Ore} (with no modification, apart from the conventions) yield that $(\qQ(T),\cdot)$ is a semigroup. 

It remains to be proven that $\qQ(T)$ is a heap. In fact, by the Ore condition we may assume that all fractions in the definition of the ternary operation \eqref{frac.heap}  on $\qQ(T)$ have common denominator, so that 
\begin{equation}\label{frac.common}
\left[\frac{a}{b},\frac{a'}{b},\frac{a''}{b}\right] = \frac{[\beta a,\beta a',\beta a'']}{\beta b},
\end{equation}
since in this case we can choose $\beta:=\beta_1=\beta_2=\beta_3$. Thus suffices it to prove that \eqref{frac.common} is well-defined, as then the heap axioms for $T$ will imply the corresponding axioms for the derived operation \eqref{frac.common}. We proceed in two steps. At first, we show that the formula \eqref{frac.common} does not depend on the choice of $\beta$; in the second stage we will prove that  it is also independent of the choice of the representatives $a,b$ for the class $\frac{a}{b}$. 

Choose another element $s \in T^{\Abs}$ such that 
\[
\left[\frac{a}{b},\frac{a'}{b},\frac{a''}{b}\right] = \frac{[s a, s a',s a'']}{s b}.
\]
There exist $g,g'\in T^{\Abs}$ such that $g\beta b=g's b$, 
which implies
\[
g\beta =g's,
\]
since $T$ is a domain. Therefore,
\[
g[\beta a,\beta a',\beta a'']=g'[s  a,  s a', s a''], \quad
g\beta b  = g's b.
\]
Consequently,
\[
\frac{[\beta a,\beta a',\beta a'']}{\beta b} = \frac{[s  a,  s a', s a'']}{s  b},
\]
which shows the independence of the formula \eqref{frac.common} of the the choice of $\beta$.

To prove that the ternary operation \eqref{frac.heap}  does not depend on the choice of the representatives  in each equivalence class, let  $(b,a),(b',a'),(b'',a''),(d,c),(d',c'),(d'',c'')\in T^\Abs\times T$ be such that 
$$
\frac{a}{b}=\frac{c}{d},\;\frac{a'}{b'}=\frac{c'}{d'},\; \frac{a''}{b''}=\frac{c''}{d''},
$$
and consider
\begin{equation}\label{frac.equ}
\left[\frac{a}{b},\frac{a'}{b'},\frac{a''}{b''}\right]=\frac{[\beta_1 a,\beta_2a',\beta_3a'']}{\beta_1 b},\qquad  \left[\frac{a}{b},\frac{a'}{b'},\frac{c''}{d''}\right]=\frac{[s_1 a,s_2a',s_3c'']}{s_1 b},
\end{equation}
for suitable $\beta_1,\beta_2, \beta_3,s_1,s_2,s_3\in T^\Abs$. Then there exist $g,g'\in T$, such that 
$$
g\beta_1 b=g\beta_2b'=g\beta_3b''=g's_1b=g's_2b'=g's_3d'',
$$ 
and, since $T$ is a domain,
$$
g\beta_1=g's_1,\;g\beta_2=g's_2.
$$
Thus both fractions in the equation \eqref{frac.equ} are equal if and only if $g\beta_3a''=g's_3c''$. Observe, however, that since $g's_3d''=g\beta_3b''$, $g\beta_3a''=g's_3c''$ as $\frac{a''}{b''}=\frac{c''}{d''}$. Therefore,
 $$
 \left[\frac{a}{b},\frac{a'}{b'},\frac{a''}{b''}\right]=\left[\frac{a}{b},\frac{a'}{b'},\frac{c''}{d''}\right].
 $$
The remaining equalities
$$
\left[\frac{a}{b},\frac{a'}{b'},\frac{c''}{d''}\right]=\left[\frac{a}{b},\frac{c'}{d'},\frac{c''}{d''}\right]\, \, \, \text{and} \, \, \, \left[\frac{a}{b},\frac{c'}{d'},\frac{c''}{d''}\right]=\left[\frac{c}{d},\frac{c'}{d'},\frac{c''}{d''}\right],
$$
are proven in a similar way. This completes the proof that  
the definition of the ternary operation \eqref{frac.heap} does not depend on the choice of representatives.

 Finally, observe that if $a\in \Abs(T)$ then the class $\frac{a}{b}$ is an absorber and it is obviously unique. One can easily check that the class $\frac{b}{b}$ for $b\in T^{\Abs}$ is a neutral element of $(\qQ(T)^{\Abs},\cdot)$ and that if $a\not \in T^{\Abs}$ then $\frac{a}{b}$ is a two-sided inverse to $\frac{b}{a}$. Thus $(\qQ(T)^{\Abs},\cdot)$ is a group. This completes the proof of the theorem.
\end{proof}

From the fact that one can find a common denominator to any system of fractions one can observe that additional properties of $T$ are carried over to $\qQ(T)$.

\begin{proposition}\label{prop:loc}
Let $T$ be a regular pre-truss.
\begin{zlist}
\item If $T$ is  Abelian, then so is $\qQ(T)$.
\item If $T$ is a near-truss, then $\qQ(T)$ is a near-truss.
\item If $T$ is a skew truss, then $\qQ(T)$ is a skew-truss.
\end{zlist}
\end{proposition}
\begin{proof}
It is sufficient to consider heap operations of fractions with a common denominator, that is, those given by the formula \eqref{frac.common}.
Statement (1) follows immediately from \eqref{frac.common}.

If $T$ is a near-truss, then
$$
\left[\frac{a}{b},\frac{a'}{b},\frac{a''}{b}\right] = \frac{[\beta a,\beta a',\beta a'']}{\beta  b} = \frac{\beta [ a,a', a'']}{\beta b} = \frac{[a,a',a'']}{b}.
$$
 Take any $\frac{c}{d}, \frac{a}{b},\frac{a'}{b},\frac{a''}{b}\in \qQ(T)$ and $\gamma,\gamma'\in T^\Abs$ such that $\gamma b = \gamma' c$, and compute
$$
\begin{aligned}
\frac{c}{d}\cdot \left[\frac{a}{b},\frac{a'}{b},\frac{a''}{b}\right] &= \frac{c}{d}\cdot\frac{[a,a',a'']}{b}
=\frac{\gamma[a,a',a'']}{\gamma' d} = \frac{[\gamma a,\gamma a', \gamma a'']}{\gamma' d}\\
&=  \left[\frac{\gamma a}{\gamma' d},\frac{\gamma a'}{\gamma' d}, \frac{\gamma a''}{\gamma' d}\right] = \left[\frac{c}{d}\cdot \frac{a}{b},\frac{c}{d}\cdot \frac{a'}{b}, \frac{c}{d}\cdot \frac{a''}{b}\right].
\end{aligned}
$$ 
Hence the left distributive law holds, and this proves statement (2).

To prove (3) we take  $\frac{c}{d}, \frac{a}{b},\frac{a'}{b},\frac{a''}{b}\in \qQ(T)$ and $\gamma,\gamma ' \in T^{\Abs}$ such that   $\gamma d=\gamma ' [a,a',a'']$. Then 
$$
\left[\frac{a}{b},\frac{a'}{b},\frac{a''}{b}\right]\cdot \frac{c}{d}=\frac{[ a,a',a'']}{b}\cdot \frac{c}{d}=\frac{\gamma c}{\gamma'b}.
$$
On the other hand, using the definitions \eqref{frac.heap} and \eqref{frac.mon}  and the right distributivity in $T$, we obtain
$$
\left[\frac{a}{b}\cdot\frac{c}{d},\frac{a'}{b}\cdot\frac{c}{d},\frac{a''}{b}\cdot\frac{c}{d} \right]=
\left[\frac{\gamma_1c}{\gamma_1' b},\frac{\gamma_2c}{\gamma_2'b},\frac{\gamma_3c}{\gamma_3 'b}\right]=\frac{[s_1\gamma_1,s_2\gamma_2,s_3\gamma_3]c}{s_1\gamma_1 'b},
$$
where $s_1,s_2,s_3,\gamma_{1},\gamma_2,\gamma_3,\gamma_{1}',\gamma_2',\gamma_3'\in T^{\Abs}$ are such that 
\begin{equation}\label{r.s.gam}
\gamma_1'a=\gamma_1b''',\; \gamma_2'a'=\gamma_2 b''',\; \gamma_3'a''=\gamma_3b''', \;  s_1\gamma_1 '=s_2\gamma_2 '=s_3\gamma_3 '.
\end{equation}
Let $h,h'\in T^{\Abs}$ be such that 
\begin{equation}\label{h.gam}
h\gamma'=h's_1\gamma_1'.
\end{equation}
Then, using the distributive laws in $T$, \eqref{r.s.gam} and \eqref{h.gam}, we find
$$
\begin{aligned}
 h\gamma d& =h\gamma' [ a,a',a'']=[h\gamma 'a,h\gamma 'a',h\gamma 'a'']=[h's_1\gamma_1'a,h's_1\gamma_1'a',h's_1\gamma_1'a'']
\\ &=h'[s_1\gamma_1'a,s_2\gamma_2'a',s_3\gamma_3'a'']=h'[s_1\gamma_1d,s_2\gamma_2d,s_3\gamma_3d]=
 h'[s_1\gamma_1,s_2\gamma_2,s_3\gamma_3]d.
 \end{aligned}
$$
The right cancellation property yields
$$
 h\gamma=h'[s_1\gamma_1,s_2\gamma_2,s_3\gamma_3],
 $$
which in view of \eqref{h.gam} implies that
$$
\frac{\gamma c}{\gamma' b}=\frac{[s_1\gamma_1,s_2\gamma_2,s_3\gamma_3]c}{s_1\gamma_1 'b}\, .
$$
Therefore, also the right distributive law holds in the near-truss $\qQ(T)$.
\end{proof}

The construction of the truss of quotients is universal in the following sense.

\begin{proposition}\label{prop.loc.univ}
Let $T$ be a regular pre-truss. Then
\begin{zlist}
\item For any   $b\in T^{\Abs}$,
$$
\iota_b:T\lra \qQ(T), \qquad a\lto \frac{ba}{b},
$$
is a monomorphism of semigroups, and it is a monomorphism of trusses provided $T$ is a near- or skew-truss.
\item If $T$ is a unital pre-truss then  $\iota_1$ is a monomorphism of unital trusses. Furthermore, for  any brace-type near-truss $B$ and any unital truss homomorphism $f: T\lra B$, there exists a  unique unital truss homomorphism $\hat{f}: \qQ(T)\lra B$ rendering commutative the following diagram:
$$
\xymatrix{T \ar[rr]^-{\iota_1} \ar[rd]_f && \qQ(T) \ar[ld]^-{\hat{f}} \\
& B. & }
$$
\end{zlist}
\end{proposition}
\begin{proof} (1) Since $T$ is regular, $\iota_b$ is an injective map.  For all $a,a'\in T$, 
$$
\iota_b \left(aa'\right) = \frac{ba a'} {b} \quad \& \quad \iota_b \left(a\right)\cdot \iota_b \left(a'\right) = \frac{ba}{b}\cdot \frac {ba'}{b} =  \frac{\gamma ba'}{\gamma' b},
$$
where $\gamma, \gamma'$ are such that $\gamma b=\gamma' ba$. Take any $\beta,\beta'\in T$ such that $\beta b = \beta'\gamma'b$. Then
$$
\beta baa' = \beta'\gamma'b aa' = \beta'\gamma b a',
$$
which means that $\iota_b \left(aa'\right) = \iota_b \left(a\right)\cdot \iota_b \left(a'\right)$ as required.

In the case of a near- or skew-truss, that  $\iota_b$ is a homomorphism of trusses follows by  \eqref{frac.common}  and the left distributive law.

(2) The monomorphism of semigroups $\iota_1$ preserves the heap operation since $1$ is the multiplicative identity in $T$. 

Assume that $f:T\to B$ is a unital homomorphism of trusses and, for all fractions $\frac{a}{b}\in \qQ(T)$, define
$$
\hat{f} : \qQ(T)\lra B,\qquad \frac{a}{b}\lto f(b)^{-1}f(a).
$$
This is well defined since two fractions $\frac{a}{b}$ and $\frac{a'}{b'}$ are identical if and only if there are $\beta,\beta'$ such that $\beta a = \beta' a'$ and $\beta b = \beta' b'$, in which case
$$
\begin{aligned}
\hat{f}\left(\frac{a}{b}\right) &=  f(b)^{-1}f(a) =f(b)^{-1}f(\beta)^{-1} f(\beta) f(a)\\
&= f(\beta b )^{-1}f(\beta a ) = f(\beta' b' )^{-1}f(\beta' a' ) =f(b')^{-1}f(a') =\hat{f}\left(\frac{a'}{b'}\right),
\end{aligned}
$$
by the multiplicativity of $f$. By the same token,
for all $\frac{a}{b}, \frac{a'}{b'} \in \qQ(T)$, 
$$
\hat{f}\left(\frac{a}{b}\cdot \frac{a'}{b'}\right) = \hat{f}\left(\frac{\gamma a'}{\gamma' b}\right) = f(\gamma' b)^{-1}f(\gamma a')= f(b)^{-1}f(\gamma')^{-1} f(\gamma) f(a'),
$$
where $\gamma,\gamma'\in T$ are such that $\gamma b' =\gamma' a$. Applying $f$ to both sides of this equality and using the multiplicative property to $f$ we obtain 
$$
f(\gamma')^{-1} f(\gamma) = f(a) f(b')^{-1},
$$
and hence
$$
\hat{f}\left(\frac{a}{b}\cdot \frac{a'}{b'}\right) =  f(b)^{-1}f(a) f(b')^{-1} f(a') = \hat{f}\left(\frac{a}{b}\right) \hat{f}\left( \frac{a'}{b'}\right)  ,
$$
that is $\hat{f}$ is a homomorphism of multiplicative groups. To check that  $\hat{f}$ is a heap morphism it is enough to consider fractions with a common denominator and then
$$
\begin{aligned}
\hat{f}\left(\left[\frac{a}{b}, \frac{a'}{b}, \frac{a''}{b}\right]\right) &= f(b)^{-1}\left[f(a),f(a') f(a'')\right]\\
 &= \left[f(b)^{-1}f(a),f(b)^{-1}f(a') f(b)^{-1}f(a'')\right]\\
 & = \left[\hat{f}\left(\frac{a}{b}\right), \hat{f}\left(\frac{a'}{b}\right), \hat{f}\left(\frac{a''}{b}\right)\right], 
\end{aligned}
$$
by the fact that $f$ is a heap homomorphism and the left distributive law in $B$. That $\hat{f}\circ \iota_1 =f$ follows by the unitality of $f$.

Suppose that there exists a unital truss homomorphism $\hat{g}: \qQ(T)\lra B$ such that $\hat{g}\circ \iota_1 =f$. Note that 
\begin{equation}\label{frac.a.b}
\frac{a}{b} = \frac{1}{b}\cdot \frac{a}{1}.
\end{equation}
 In particular,
$$
1 = \hat{g}\left(\frac{1}{1}\right) =  \hat{g}\left(\frac{1}{b}\cdot \frac{b}{1}\right) = \hat{g}\left(\frac{1}{b}\right) f(b),
$$
where the last equality follows by the splitting assumption $\hat{g}\circ \iota_1 =f$. Hence $\hat{g}\left(\frac{1}{b}\right)= f(b)^{-1}$ and the equality $\hat{g} = \hat{f}$ follows by the multiplicativity of $\hat{g}$ and equations \eqref{frac.a.b}.
\end{proof}

The following corollary provides one with the method of constructing skew braces, which might be considered as one of the main results of this paper.

\begin{corollary}\label{cor:qb}
If $T$ is a regular near-truss without an absorber, then $\qQ(T)$ is a brace-type near-truss, that is, for all $b\in T$, the retract of $\qQ(T)$ at ${\frac{b}{b}}$ with the product \eqref{frac.mon}  is a skew brace.
\end{corollary}
\begin{proof}
Observe that if $T$ has no absorbers then $\qQ(T)$ has no absorbers either. Indeed, suppose that there exists $\frac{a}{b}\in \qQ(T)$ such that, for all $\frac{c}{d}\in \qQ(T)$,  $\frac{c}{d}\cdot\frac{a}{b}=\frac{a}{b}$. Since $T$ has no absorbers, it has at least two elements, and hence, in particular we may consider $c\neq d$. Then 
there exist $\gamma,\gamma '\in T$, such that $\frac{\gamma a}{\gamma' d}=\frac{\gamma a}{\gamma b}$ and $\gamma 'c=\gamma b$. Thus  $\frac{\gamma a}{\gamma' d}=\frac{\gamma a}{\gamma' c}$, so there exist $\beta, \beta' \in T$ such that  $\beta\gamma'd=\beta'\gamma'c$ and $\beta\gamma a=\beta' \gamma a$. By regularity, $\beta=\beta'$ and $c=d$, which is the required contradiction. Therefore, $\frac{a}{b}$ is not an absorber for all $a,b\in T$. Now, since $\qQ(T)$ is a group with multiplication  and identity $\frac{b}{b}$, the retract of $\qQ(T)$ in $\frac{b}{b}$ is a skew brace by \cite[Remark~3.13]{Brz:par}.
\end{proof}

Note in passing that if $T$ satisfies the same assumptions as in Corollary~\ref{cor:qb}, but there exists an absorber in $T$, then $\qQ(T)$ is associated with a near-field.

\begin{example}
Let us consider $2\ZZ+1$. It is a domain satisfying the Ore condition, thus it is a regular truss and we can localise it in itself. Since $2\ZZ+1$ is commutative, the construction is much simpler than the one presented in the proof of Theorem \ref{thm:loc}. One can easily check that $\qQ(2\ZZ+1)=\frac{2\ZZ+1}{2\ZZ+1}:=\left\{\frac{2p+1}{2q+1}\;|\; p,q\in \ZZ\right\}$. The two-sided brace associated with this truss is the retract in $1$, i.e.\ the triple $(\qQ(2\ZZ+1),[-,1,-],\cdot)$.

Similarly, the truss $O(x)$ of integer polynomials with coefficients summing up to odd numbers considered in Example~\ref{ex.poly} is regular with no absorbers, and hence it can be localised to a brace-type truss of the following rational functions
$$
\qQ(O(x))=\frac{O(x)}{O(x)}:=\left\{\frac{p(x)}{q(x)}\;|\; p(x),q(x)\in O(x)\right\}.
$$

As a yet another example we can consider the truss $O(i)$ constructed as a special case of Example~\ref{ex.poly}. Again this is a commutative domain satisfying the Ore condition and with no absorbers, and hence 
$$
\begin{aligned}
\qQ(O(i)) &= \left\{\frac{m+ ni}{p +qi}\;|\; \mbox{$m+n$ and $p+q$ are odd integers}\right\}
\\
&= \left\{ \frac{m}{2p+1}+ \frac{n}{2q+1}i\;|\; p,q\in \ZZ, \; \mbox{$m+n$ is an odd integer}\right\}\, .
\end{aligned}
$$
\end{example}

The example of odd fractions described above is a special case of a more general construction.

\begin{example}\label{ex.matrix}
Let $T_n(\ZZ)$ denote the set of all $n\times n$-matrices over $\ZZ$ with odd entries on the diagonal and even off diagonal entries. That is,
$$
T_n(\ZZ) = \left\{ (a_{ij})_{i,j=1}^n \; |\; a_{ii}\in 2\ZZ +1\; \&\;  a_{ij}\in 2\ZZ,\; i\neq j\right\}.
$$
\begin{zlist}
\item $T_n(\ZZ)$ endowed with the matrix multiplication and the standard heap operation $[\bfa,\bfb,\bfc] = \bfa -\bfb +\bfc$ is a unital regular truss with no absorbers. 
\item The brace-type truss of fractions $\qQ(T_n(\ZZ))$ can be identified with the set $T_n(\QQ)$ of  $n\times n$-matrices over the rational numbers with diagonal entries made by the odd fractions (that is, fractions of both the numerator and denominator odd, $\qQ(2\ZZ+1)$) and with fractions with even numerator and odd denominator as off-diagonal entries. That is,
$$
\qQ(T_n(\ZZ))\cong T_n(\QQ) := \left\{ (q_{ij})_{i,j=1}^n \; |\; q_{ii}\in \frac{2\ZZ +1}{2\ZZ +1}\; \&\;  q_{ij}\in \frac{2\ZZ}{2\ZZ +1},\; i\neq j\right\}.
$$
\end{zlist}
\end{example}

It is clear that the set $T_n(\ZZ)$ is closed under the described heap operation. That it is closed also under the matrix multiplication follows from an observation that in the product formula for the off-diagonal entries the sum involves the products of numbers of which at least one is even, while for the diagonal entry there is a single odd summand made out of the product of matching diagonal entries. Obviously $T_n(\ZZ)$ has no absorber, as the zero matrix is not an element of $T_n(\ZZ)$. Since the identity matrix has the prescribed form, $T_n(\ZZ)$ is unital.
The other statements of Example~\ref{ex.matrix} can be justified by the following (elementary) lemma.

\begin{lemma}\label{lem.matrix}
For all $\bfa\in T_n(\ZZ)$,
\begin{rlist}
\item The determinant $\det (\bfa) $ is an odd number.
\item The matrix of cofactors $\bar{\bfa}$ of $\bfa$ and hence also its transpose  $\bar{\bfa}^t$ are elements of $T_n(\ZZ)$.
\end{rlist}
\end{lemma}
\begin{proof}
Let $\bfa_{i,j}$ denote the matrix obtained from $\bfa$ by removing of the $i$-th row and the $j$-th column. Note that $\bfa_{i,i}\in T_{n-1}(\ZZ)$ and that $\bfa_{i,j}$, $i\neq j$ has one row of even numbers.

The first statement is proven by induction on the size $n$ of matrices. For $n=1$ the statement is obviously true. Assuming that the statement is true for $k$ we calculate the determinant of $\bfa \in T_{k+1}(\ZZ)$ by expanding by the first row. Since  $\bfa_{1,1}$  is an element of $T_k(\ZZ)$, $\det (\bfa_{1,1})$ is odd by inductive assumption. In the expansion of $\det (\bfa)$ this is multiplied by the first entry $a_{11}$ of $\bfa$ and thus it gives an odd number. All the remaining summands involve products of  other entries of the first row, which are even. Hence the sum of all terms in the expansion is odd as required.

The diagonal entries of $\bar{\bfa}$ are given by $\det (\bfa_{i,i})$ which are odd by statement (i). Off-diagonal entries $(-1)^{i+j}\det (\bfa_{i,j})$ are even since one row of  each of $\bfa_{i,j}$, $i\neq j$ consists entirely of even numbers. The transposition statement is obvious.
\end{proof}

With this lemma at hand we can now  prove that $T_n(\ZZ)$ is a domain satisfying the Ore condition. Since we can embed $T_n(\ZZ)$ into a ring of matrices, the statement $\bfa \bfb = \bfa \bfc$, for some $\bfa,\bfb,\bfc \in T_n(\ZZ)$ is equivalent to the statement that $\bfa (\bfb - \bfc) =0$, hence
$$
0 = \bfa (\bfb - \bfc) = \bar{\bfa}^t \bfa (\bfb - \bfc) = \det (\bfa) (\bfb -\bfc),
$$
which implies that $\bfb=\bfc$, since $\det(\bfa) \neq 0$ by Lemma~\ref{lem.matrix}(i). The regularity of the other side of each $\bfa\in T_n(\ZZ)$ can be proven in a symmetric way.

To prove the Ore condition we take any $\bfa, \bfb\in T_n(\ZZ)$ and set
$$
\mathbf{r} = \bfa \bar{\bfb}^t \quad \& \quad \mathbf{s} = \det(\bfb) \mathbf{1}.
$$
Both these matrices are elements of $T_n(\ZZ)$ by Lemma~\ref{lem.matrix}, and they satisfy the Ore condition 
$\mathbf{s} \bfa = \mathbf{r} \bfb$. Hence, $T_n(\ZZ)$ is a left regular (in fact also right regular by similar arguments) truss.

For any element $\bfq\in T_n(\QQ)$ we write $q$ for the product of all denominators in entries of $\bfq$. This is an odd number and thus obviously $q\bfq \in T_n(\ZZ)$. In particular, in view of Lemma~\ref{lem.matrix}, $\det(q\bfq )$ is an odd number and its matrix of cofactors is an element of $T_n(\ZZ)$. This in turn implies that the inverse of $\bfq$ is an element of $T_n(\ZZ)$ divided by an odd number, hence an element of $T_n(\QQ)$. Consequently, $T_n(\QQ)$ is group with respect to multiplication of matrices. In order to identify $T_n(\QQ)$ with the truss of fractions $\qQ(T_n(\ZZ))$ we will explore the universal property described in Proposition~\ref{prop.loc.univ}(2).  Thus consider a brace-type skew truss $B$ and a homomorphism of unital trusses $f: T_n(\ZZ)\lra B$ and set
$$
\hat{f}: T_n(\QQ)\lra B, \qquad \bfq \lto f(q\mathbf{1})^{-1}f(q\bfq).
$$
Note that this definition does not depend on the way the fractions in $\bfq$ are represented, as the multiplication of the numerator and a denominator of an entry by a common (odd) factor results in multiplying both $q$ and $\bfq$ by the same factor which will cancel each other out in the formula for $\hat{f}$, by the multiplicative property of $f$. Since $q\mathbf{1}$ is a central element in $T_n(\ZZ)$, $f(q\mathbf{1})^{-1}$ is central in the image of $f$ and, combined with the multiplicative property of $f$ this implies that $\hat{f}$ is a homomorphism of (multiplicative) groups. That $\hat{f}$ is a homomorphism of heaps follows by the distributivity. Obviously, $\hat{f}\circ \iota_{\mathbf{1}} = f$ and is  a unique such morphisms. By the uniqueness of universal objects, $T_n(\QQ)$ is isomorphic to the truss of fractions $\qQ(T_n(\ZZ))$.

\section*{Acknowledgement}
The research of Tomasz Brzezi\'nski  is partially supported by the National Science Centre, Poland, grant no. 2019/35/B/ST1/01115.

\end{document}